\documentclass[12pt]{amsart}
\usepackage{amsmath}
\usepackage{geometry}
\usepackage{amsthm}
\usepackage{amsfonts}
\usepackage{comment}
\usepackage{amssymb}
\usepackage{eufrak}
\usepackage{comment}
\usepackage{hyperref}
\usepackage{mathtools}

\numberwithin{equation}{section}
\usepackage{breqn}

\newtheorem{Theorem}{Theorem}[section]
\newtheorem{Corollary}[Theorem]{Corollary}
\newtheorem{Lemma}[Theorem]{Lemma}

\theoremstyle{definition}

\theoremstyle{remark}
\newtheorem*{remark}{Remark}

\usepackage{cite}


\title{Traces of singular values of Hauptmoduln}

\author[Lea Beneish]{Lea Beneish}
\address{Department of Mathematics, Indiana University, Rawles Hall, 831 E. 3rd Street, Bloomington IN 47405}
\email{lbeneish@indiana.edu}

\author[Hannah Larson]{Hannah Larson}
\address{Department of Mathematics, Harvard University, One Oxford Street, Cambridge MA 02138}
\email{hannahlarson@college.harvard.edu}

\begin{document}

\maketitle

\begin{abstract}In an important paper, Zagier proved that certain half-integral weight modular forms are generating functions for traces of polynomials in the $j$-function. It turns out that Zagier's work makes it possible to algorithmically compute Hilbert class polynomials using a canonical family of modular forms of weight $\frac{3}{2}$. We generalize these results and consider Haupmoduln for levels $1, 2, 3, 5, 7,$ and $13$.  We show that traces of singular values of polynomials in Haupmoduln are again described by coefficients of half-integral weight modular forms. This realization makes it possible to algorithmically compute class polynomials.
\end{abstract}

\section{Introduction and Statement of Results}

The modular invariant $j$ is defined by 
\[j(z):=\frac{E_4(z)^3}{\Delta(z)}= q^{-1}+744+ 196884q+ \ldots,\]
where $q:=e^{2\pi iz},$ $E_4(z)$ is the Eisenstein series of weight $4$, and $\Delta(z)$ is the modular discriminant function. The values of the $j$-function at CM points are known to be algebraic integers. Let $\mathcal{Q}_D$ be the set of positive definite binary quadratic forms of discriminant $-D$, and let $\mathcal{Q}_D/\Gamma_0(1)$ denote equivalence classes under the action of the modular group $\Gamma_0(1) = SL_2(\mathbb{Z})$. Given a binary quadratic form $Q$, we let $\alpha_Q$ denote the unique root of $Q$ in the upper half-plane. For $-D$ a fundamental discriminant, the \textit{Hilbert class polynomial}
\begin{equation}
\mathcal{H}_D(x):=\prod_{Q \in \mathcal{Q}_D/\Gamma_0(1)} (x-j(\alpha_Q))
\end{equation}
is a monic, irreducible polynomial whose splitting field is the Hilbert class field of $\mathbb{Q}(\sqrt{-D})$. 

The classically difficult problem of computing $\mathcal{H}_D(x)$ can be answered by
the recent work of Zagier in \cite{traces}. Let  $J(z) := j(z) - 744,$ and let $T_{\nu}$ denote the Hecke operator of index $\nu$. In revisiting Borcherd's infinite product formulas, Zagier shows that there exist half-integal weight modular forms whose coefficients describe modified traces of the forms $\nu J(z)\vert T_{\nu}$. Since $\nu J(z)\vert T_{\nu}$ is expressible as a degree $\nu$ polynomial in $j(z)$, this makes computing the Hilbert class polynomial an exercise in diagonalizing to find power sums and applying the Newton-Girard formulae to recover symmetric polynomials.

A natural problem is to compute the minimal polynomial of a Hauptmodul of level $N$ evaluated at Heegner points of level $N$. When the congruence subgroup $\Gamma_0(N)$ has genus zero, a \textit{Hauptmodul of level $N$} is a generator for the field of modular functions, chosen to have a simple pole at the cusp at infinity, and is unique up to a constant. The $j$-function is a Hauptmodul for level $1$. Recall that the Dedekind-eta function is defined by
\[\eta(z):= \Delta(z)^{1/24} = q^{1/24}\prod_{n=1}^{\infty}(1 - q^n).\]
The following table lists a Hauptmodul $j^{(N)}(z)$ in terms of an eta-quotient for the levels $N > 1$ where they are defined.

\begin{center}
\begin{table}[h]  \caption{Hauptmoduln as eta-quotients} \label{table}
\begin{tabular}{| c | c | c | c | c | c | c | c| }
\hline
$N$ & 2 & 3 & 4 & 5 & 6 & 7 & 8  \\
\hline
$j^{(N)}(z)$ & ${{\eta(z)}^{24}\over \eta(2z)^{24}}$ & ${\eta(z)^{12}\over \eta(3z)^{12}}$ & ${\eta(z)^8\over \eta(4z)^8}$ & ${\eta(z)^6\over \eta(5z)^6}$ & ${\eta(2z)^3\eta(3z)^9\over \eta(z)^3\eta(6z)^9}$ & ${\eta(z)^4\over \eta(7z)^4}$ & ${\eta(z)^4\eta(4z)^2\over \eta(2z)^2\eta(8z)^4}$\\
\hline
\end{tabular}

\vspace{.2in}
\begin{tabular}{ | c | c | c | c | c | c | c |}
\hline
$N$ & 9 & 10 & 12 & 13 & 16 & 25\\
\hline
$j^{(N)}(z)$ & ${\eta(z)^3\over \eta(9z)^3}$ &${\eta(2z)\eta(5z)^5\over \eta(z)\eta(10z)^5}$ & ${\eta(4z)^4\eta(6z)^2\over \eta(2z)^2\eta(12z)^4}$ & ${\eta(z)^2\over \eta(13z)^2}$ & ${\eta(z)^2\eta(8z)\over \eta(2z)\eta(16z)^2}$ & ${\eta(z)\over \eta(25z)}$\\
\hline
\end{tabular}
\end{table}
\end{center}
We will write $J^{(N)}(z)$ for the normalized Hauptmodul of level $N$ with constant term equal to $0$.

Let $\mathcal{Q}_D^N$ be the set of binary quadratic forms of discriminant $-D$ corresponding to Heegner points of level $N$ (those forms in which the coefficient of $x^2$ is divisible by $N$). We define the class polynomials
\begin{equation}
\mathcal{H}_D^{(N)}(x):=\prod_{Q \in \mathcal{Q}_D^N/\Gamma_0(N)} (x-j^{(N)}(\alpha_Q)).
\end{equation}
In \cite{MP}, Miller and Pixton generalize Zagier's traces, showing that there exist modular forms of half-integal weight whose coefficients describe the ``traces" of certain integral weight Poincar\'e series. Generically, such Poincar\'e series will have transcendental coefficients, so the work of Zagier is a very special result.
We apply their results to a special family of polynomials in $j^{(N)}(z)$ to give explicit formulas for algebraic traces, thus determining the $\mathcal{H}_D^{(N)}(x)$. These polynomials are constructed using a generalization of the generating function given in Corollary 4 of ~\cite{kaneko}. Let $P_{\nu}^{(N)}(x)$ to be the polynomial defined by
\begin{equation} \label{defP}
\frac{j^{(N)'}(z)}{x - j^{(N)}(z)} = \sum_{\nu = 0}^{\infty} P_{\nu}^{(N)}(x)q^{\nu}.
\end{equation}
When $N = 1$, Asai, Kaneko, and Ninomiya ~\cite{kaneko} show that $P_{\nu}^{(1)}(j(z)) = \nu J(z)\vert T_{\nu}$. Our first task is to understand part of a larger framework that connects these polynomials to the Hecke algebra. For the levels $N$ on which these polynomials are defined, we show that $P_{\nu}^{(N)}(j^{(N)}(z))$ can be expressed as a linear combination of Hauptmoduln of levels dividing $N$ and $\nu$ hit with combinations of Hecke operators.

\begin{Theorem} \label{poly}
For each $N \in \{1, 2, 3, 4, 5, 6, 7, 8, 9, 10, 12, 13, 18, 25\}$, let $j^{(N)}(z)$ be the Hauptmodul given in Table \ref{table} and let $J^{(N)}(z)$ be its normalization such that the constant term is zero. For each positive integer $\nu$,
\[P_{\nu}^N(j^{(N)}(z))=\sum_{d|(\nu,N)} \frac{\nu}{d} J^{(N)}(z) \mid T_{\frac{\nu}{d}}V_d,\]
where $T$ and $V$ are the standard Hecke operators defined in Section \ref{heckesec}.
\end{Theorem}

This theorem plays a central role in computing class polynomials of $j^{(N)}(z)$. To this end, it will be important to introduce sequences of half-integal weight modular forms whose coefficients describe the traces of these polynomials. 
For a positive integer $D \equiv 0, 3 \pmod 4$, we define
\begin{equation}
\text{Tr}_{\nu}^{(N)}(D) := \sum\limits_{Q\in\mathcal{Q}_D^N/\Gamma_0(N)} \frac{P_{\nu}^{(N)}(j^{(N)}(\alpha_Q))}{w_{Q, N}},
\end{equation}
where $w_{Q, N}$ is the order of the stabilizer of a form $Q$ under the action of $\Gamma_0(N)$.
The following theorem is a more explicit version of Theorem 1.1 of \cite{MP}.

\begin{Theorem}  \label{traces}
Let $N \in \{1, 3, 5, 7, 13\}$, and set $\varpi(N) =  \#(\Gamma_0(1)/\Gamma_0(N))$. Let $\widetilde{b}_N(-m;n)$ denote the coefficient of $q^n$ in the weakly holomorphic modular forms $\widetilde{F}_N(-m;z)$ of weight $\frac{3}{2}$ defined in Section \ref{forms}. Then, for each positive integer $D$ with $D \equiv 0, 3 \pmod 4$ and any positive integer $\nu$, we have
\[\mathrm{Tr}_{\nu}^{(N)}(D) =-\nu \sum\limits_{d|\nu} \frac{1}{d}\left(\widetilde{b}_{\frac{N}{(N,d)}}(\tfrac{-\nu^2}{d^2};D) - \frac{24}{\varpi(N)}H_1(D)\right) - H_N(D)\mathfrak{c}_{N, \nu},\]
where $H_{N}(D)$ and $\mathfrak{c}_{N, \nu}$ are constants defined in Section \ref{forms}.
\end{Theorem}

These sequences of half-integal weight modular forms $\widetilde{F}_N(-m;z)$ are well-defined and can be recovered recursively from the two seed functions $\widetilde{F}_N(0;z)$ and $\widetilde{F}_N(-1; z)$ (except when $N = 1$ where $\widetilde{F}_1(0;z) = 0$, in which case $\widetilde{F}_1(-4;z)$ is required). The reader should consult the Appendix for a description of the seed functions in the levels we consider.

Using Theorems \ref{poly} and \ref{traces}, we obtain an algorithm for computing the $\mathcal{H}_D^{(N)}(x)$ for fundamental discriminants $-D$.

\begin{Theorem}
For $N \in \{1, 3, 5, 7, 13\}$ and $-D$ a fundamental discriminant, the algorithm given in Section \ref{algsec} computes $\mathcal{H}_D^{(N)}(x)$.
\end{Theorem}

\begin{remark}
In \cite{Suth}, Eagle and Sutherland give an algorithm for computing these class polynomials using elliptic curves.
\end{remark}

\begin{remark}
Gross provides an interesting approach to traces of singular moduli in \cite{Gross} and it is possible that these methods could also be used in determining singular values of Hauptmoduln.
\end{remark}

This paper is organized as follows. In Section \ref{heckesec}, we define relevant operators and prove a number of lemmas describing their effects on Fourier expansions. We apply these results to prove Theorem \ref{poly} in Section \ref{proofpoly}. In Section \ref{forms}, we first recall the results of Miller and Pixton, and then introduce a family of weakly holomorphic modular forms of weight $\frac{3}{2}$ and prove Theorem \ref{traces}. In Section \ref{algsec}, we detail an algorithm for computing the class polynomials. The last section demonstrates how to apply the algorithm in an explicit numerical example.

\vspace{.1in}
\noindent
\textbf{Acknowledgements:} The authors would like to thank Professor Ken Ono for suggesting the topic and for advice and guidance throughout the process, and an anony- mous referee for useful comments on a draft of this paper. We also would like to thank Michael Griffin and Sarah Trebat-Leader for useful conversations. Both authors are also grateful to NSF for its support.

\section{Hauptmoduln and Proof of Theorem \ref{poly}}
\subsection{Hecke Operators and Atkin-Lehner Involutions} \label{heckesec} Let $M_{k}(\Gamma_0(N))$ be the space of holomorphic modular forms of weight $k$ and level $N$.
We denote by $M_k^{!}(\Gamma_0(N))$ the space of meromorphic modular forms of weight $k$ and level $N$ whose poles, if any, are supported at the cusps. Such forms are known as  \textit{weakly holomorphic} modular forms. We write $M_k^{\#}(\Gamma_0(N))$ for the subspace of $M_k^{!}(\Gamma_0(N))$ of modular forms whose only poles are supported at the cusp at infinity.

We first recall the definitions of basic operators on $\Gamma_0(N)$. For $f$ a meromorphic modular form of weight $k$, and any $\gamma = \left(\begin{array}{cc} a & b \\ c & d\end{array}\right) \in GL_2^+(\mathbb{R})$, the ``slash" operator $\vert_{k}$ is defined by
\[(f\vert_{k}\gamma)(z):=(\text{det}\gamma)^{k/2}(cz + d)^{-k}f\left(\frac{az + b}{cz + d}\right).\]
Since the weight will be clear from context, we drop the subscript $k$ and just write $f\vert \gamma$.
For a positive integer $d$, Atkin's $U$-operator is defined by
\begin{equation}
\left(\sum_{n \in \mathbb{Z}} a_nq^n\right)\vert U_d=\sum_{n \in \mathbb{Z}} a_{dn}q^n,
\end{equation}
and can be written in terms of the slash operator as
\begin{equation}
f\vert U_d = d^{\frac{k}{2} - 1}\sum_{j=0}^{d-1}f \vert \left(\begin{array}{cc} 1 & j \\ 0 & d\end{array}\right).
\end{equation}
The $V$-operator is defined by
\begin{equation}
\left(\sum_{n \in \mathbb{Z}} a_nq^n\right) \vert V_d=\sum_{n \in \mathbb{Z}}a_nq^{dn},
\end{equation}
and can be written in terms of the slash operator as
\begin{equation}
f \vert V_d = d^{-\frac{k}{2}}f \vert \left( \begin{array}{cc} p & 0 \\ 0 & 1\end{array}\right).
\end{equation}
For $p$ prime, the $p$th Hecke operator on $\Gamma_0(N)$ is defined by
\[T_p:= U_p + p^{k-1}\epsilon(p)V_p,\]
where $\epsilon(p) = 1$ for $p \nmid N$ and $0$ for $p \mid N$. The Hecke operators satisfy $T_{mn} = T_{m}T_n$ for $m$ and $n$ coprime and for $r \geq 2$, we have $T_{p^r} = T_{p^{r-1}}T_p - \epsilon(p)p^{k-1}T_{p^{r-2}}$. If $p_1, \ldots, p_n$ are the distinct primes dividing $N$, and $m = p_1^{r_1} \cdots p_n^{r_n}s$ where $(s, m) = 1$, we will often write $T_m = U_{p_1}^{r_1} \cdots U_{p_n}^{r_n}T_s$. The Hecke operators act on Fourier expansions by
\begin{equation} \label{heckeformula}
\left(\sum\limits_{n \in \mathbb{Z}}a_nq^n\right)|T_m = 
\sum_{n \in \mathbb{Z}}\left(\sum_{d \mid (m,n)}\chi(d)d^{k-1}a_{mn/d^2}\right)q^n.
\end{equation}
where $\chi(d) = 1$ if $(d, N) = 1$ and $0$ otherwise.

A direct consequence of the above equation is the following.

\begin{Lemma} \label{Tqexp}
Suppose $f \in M_k^{!}(\Gamma_0(N)$ has the Fourier expansion $f(z)=\sum a_nq^n=q^{-\nu}+O(q)$ for some positive integer $\nu$. If $m$ is any positive integer with $(m, \nu)=1$ and $(m,N)=1$, then $m^{1-k} \cdot f(z)|T_{m}$ has Fourier expansion beginning $q^{-m\nu}+O(q)$.
\end{Lemma}

For a prime divisor $p$ of $N$ for which $(p^{\alpha},\frac{N}{p^{\alpha}})=1$, the \textit{Atkin-Lehner involution at $p$} is defined to be any matrix of the form
\begin{equation}
W_{p^{\alpha}}= \left(\begin{array}{cc} p^{\alpha}x & y \\ Nz & p^{\alpha}w \end{array}\right)
\end{equation}
for integers $x, y, z, w$, such that the determinant is $p^{\alpha}$. These operators define involutions on $M_k^{!}(\Gamma_0(N))$. Products of Atkin-Lehner involutions correspond to the cusps of $\Gamma_0(N)$, and slashing by these matrices give the expansions at those cusps.

\begin{remark}
All choices of $x, y, z, w$ that satisfy the conditions on the determinant of $W_{p^{\alpha}}$ are equivalent under the action of $\Gamma_0(N)$ so this is well-defined.
\end{remark}

If $N=p_1^{\alpha_1}\cdots p_n^{\alpha_n}$, the product of the $W_{p_i^{\alpha_i}}$ is equivalent to the \textit{Fricke involution} 
\begin{equation}
W_{N}= \left(\begin{array}{cc} 0 & -1 \\ N & 0 \end{array}\right).
\end{equation}

\begin{remark}
A number of lemmas we will reference, involving the action of operators on Fourier expansions, were originally only stated for cusp forms. However, the proofs are equally valid for any weakly holomorphic modular form.
\end{remark}

The following result (Lemma 2 of ~\cite{Li}) shows that, under certain conditions, the Atkin-Lehner involutions commute with the Hecke operators $V$ and $T$.

\begin{Lemma} \label{winnie}
Let $N$ be a positive integer, and let $p$ and $p'$ be primes with $p^{\alpha}||N$. Then the following are true:
\begin{enumerate}
\item If $(p',p)=1$, then $f|V_{p'}W_{p^{\alpha}}=f|W_{p^{\alpha}}V_{p'}$ for any $f \in M_k^!(\Gamma_0(\frac{N}{p'}))$.
\item If $(p',N)=1$, then $f|W_{p^{\alpha}}T_{p'}=f|T_{p'}W_{p^{\alpha}}$ for any $f \in M_k^!(\Gamma_0(N))$. \label{WT=TW}
\end{enumerate}
\end{Lemma}

It turns out that, in some cases, the operators $W_{p^{\alpha}}$ also commute with the $U$-operator.

\begin{Lemma} \label{WandUcommute}
Let $f$ be in $M_k^!(\Gamma_0(N))$, and suppose $N=p^\alpha M$ with $(p,M)=1$. If $\ell$ is any prime dividing $N$ with $(p,\ell)=1$, then $f|U_{\ell}W_{p^\alpha}=f|W_{p^\alpha}U_{\ell}$. 
\end{Lemma}

\begin{proof} We have
\begin{align*}
f|U_{\ell}W_{p^{\alpha}} &=\frac{1}{\ell}\sum\limits_{j=0}^{{\ell}-1} f\vert\left(\begin{array}{cc} 1 & 0 \\ 0 & \ell \end{array}\right)\left(\begin{array}{cc} 1 & j \\ 0 & 1 \end{array}\right)\left(\begin{array}{cc} p^{\alpha}x & 1 \\  Nz & p^\alpha \end{array}\right) \\
&= \frac{1}{\ell}\sum\limits_{j=0}^{\ell-1} f|\left(\begin{array}{cc} p^\alpha x + Nzj & 1 +jp^\alpha\\ zN \ell & p^\alpha \ell \end{array}\right).
\end{align*}
 Since $(p,\ell)=1$ and $\ell \mid N$, for each $j$, there exists $i_j$ such that 
 \[-i_j(p^\alpha x+Nzj)+(1+jp^\alpha)=n_j\ell\]
 for some integer $n_j$. As $j$ runs over residues mod $\ell$, so does $i$.
Hence, we can write the sum as 
\begin{align*}
\frac{1}{\ell}\sum\limits_{j=0}^{\ell-1} \ &f \ \vert\left(\begin{array}{cc} p^\alpha x + Nzj & 1 +jp^\alpha-i_j(p^\alpha x+Nzj)\\  Nz\ell & p^\alpha \ell -i_jNz\ell \end{array}\right) \left(\begin{array}{cc} 1 & i_j \\ 0 & 1 \end{array}\right) \\
&=\frac{1}{\ell}\sum\limits_{j=0}^{\ell-1} f|\left(\begin{array}{cc} p^\alpha x +Nzj & n_j \\ Nz\ell & p^\alpha -i_jNz \end{array}\right) \left(\begin{array}{cc} 1 & 0 \\ 0 & \ell \end{array}\right) \left(\begin{array}{cc} 1 & i_j \\ 0 & 1 \end{array}\right)
\end{align*}
\begin{align*}
&= \frac{1}{\ell}\sum\limits_{j=0}^{\ell-1} f|W_{p^{\alpha}}\left(\begin{array}{cc} 1 & 0 \\ 0 & \ell \end{array}\right) \left(\begin{array}{cc} 1 & i_j \\ 0 & 1 \end{array}\right)=f|W_{p^{\alpha}}|U_{\ell},
\end{align*}
as desired.
\end{proof}

The next lemma allows us to recursively determine $f|U_p^aW_p$ in the case when the two operators do not commute.

\begin{Lemma} \label{thegame}
Let f be in $M_{k}^!(\Gamma_0(N))$ and let $p$ be a prime with $p||N$. Then 
\[pf|U_p^aW_p=pf|U^a_pV_p+f|U^{a-1}_pW_pV_p-f|U_p^{a-1}\]
for all integers $a \geq 1$.
\end{Lemma}

\begin{proof} Lemma $7$ of  ~\cite{AK} applied to $f|U_p^{a-1}$ says that $f|U_p^{a-1}$ is on $\Gamma_0(N)$ and that
\[pf|U_p^{a}+f|U_p^{a-1}W_p\]
 is on $\Gamma_0(\tfrac{N}{p})$. Observe that $W_p= \left(\begin{array}{cc} px & y \\ Nz & pw \end{array}\right)= \left(\begin{array}{cc} x & y \\ Nz/p & pw \end{array}\right) \left(\begin{array}{cc} p& 0 \\ 0 & 1 \end{array}\right)$. Since the first matrix in this product is in $\Gamma_0(\tfrac{N}{p})$, it fixes the above sum, and we see that
\[(pf|U_p^{a}+f|U_p^{a-1}W_p)|W_p= (pf|U_p^{a}+f|U_p^{a-1}W_p)|V_p.\]
Since $W_p$ is an involution, we find
\[pf|U_p^aW_p+f|U_p^{a-1}=pf|U_p^{a}V_p+f|U_p^{a-1}W_pV_p,\] 
demonstrating the desired identity.
\end{proof}

We will apply this lemma frequently in the following context.
\begin{Corollary} \label{poles}
If $f \in M_0^{\#}(\Gamma_0(N))$ has Fourier expansion beginning $q^{-1} + O(q)$ and $p$ is a prime with $p \mid\mid N$, then
\[p^af|U_p^aW_p=-q^{-p^{a-1}}+O(q)\]
for all integers $a \geq 1$.
\end{Corollary}
\begin{proof}
We proceed by induction on $a$. When $a = 1$, applying Lemma \ref{thegame} shows that
\[pf\vert U_pW_p = pf \vert U_pV_p + f \vert W_pV_p - f = -q^{-1} + O(q)\]
because neither $f \vert U_p$ nor $f \vert W_p$ has a pole at $i\infty$ (the operator $W_p$ sends the pole to another cusp). Hence, $f\vert U_pV_p$ and $f \vert W_pV_p$ do not have poles at $i \infty$ and neither contributes to the principal part of the expansion at $i \infty$.

Now suppose $a>1$ and $p^{a-1}f \vert U_p^{a-1}W_p = -q^{p^{a-2}} + O(q)$. Applying Lemma \ref{thegame}, we find
\[p^af\vert U_p^a W_p = p^af|U^a_pV_p+p^{a-1}f|U^{a-1}_pW_pV_p-p^{a-1}f|U_p^{a-1}.\]
Since $f\vert U_p^{a-1}$ has no pole at $i \infty$, the first and last terms on the right hand side do not contribute to the principal part of the expansion at $i\infty$. Hence,
\[p^af\vert U_p^a W_p= (-q^{p^{a-2}} + O(q)) \vert V_p = -q^{p^{a-1}} + O(q).\]
\end{proof}

Next we show that for levels $N$ divisible by a square, the normalized Hauptmodul for level $N$ is annihilated by certain $U$-operators.
\begin{Lemma} \label{zero} 
Let $N \in \{4, 8, 9, 12, 18, 25\}$ and suppose $p$ is a prime with $p^2|N$. Then $J^{(N)}(z)|U_p=0.$
\end{Lemma}

\begin{proof} The form $J^{(N)}(z) \Delta(pz)$ is holomorphic of level $Np$, so $(J^{(N)}(z) \Delta(pz))\vert U_p = J^{(N)}(z)|U_p\cdot\Delta(z)$ is a weight $12$, holomorphic modular form of level $Np$. Since this space is finite dimensional, calculating that the first few coefficients of these expansions are zero proves the claim. The dimension formulas in Theorem 1.34 of ~\cite{ken} determine how many coefficients need to be checked.
\end{proof}

The last two lemmas of this section give some useful properties of the $V$-operator in relation to the forms we will consider.

\begin{Lemma} \label{VW}
Let $f$ be in $M_k^!(\Gamma_0(N))$. If $p$ is any prime not dividing $N$, then $f|V_p$ is on $\Gamma_0(Np)$ and $f|V_pW_p=p^kf$.
\end{Lemma}

\begin{proof} We have 
\[f|V_pW_p=f| \left(\begin{array}{cc} p & 0 \\ 0 & 1 \end{array}\right)\left(\begin{array}{cc} px & y \\ Npz & pw \end{array}\right)=f| \left(\begin{array}{cc} p^2x & py \\ Npz & pw \end{array}\right)=p^kf| \left(\begin{array}{cc} px & y \\ Nz & w \end{array}\right).\]
The last matrix is in $\Gamma_0(N)$ and hence leaves $f$ fixed. Thus $f|V_pW_p=p^kf$.
\end{proof}

\begin{Lemma} \label{V}
Suppose $f$ is in $M_k^\#(\Gamma_0(N))$. If $p$ is a prime such that $p|N$, then $f|V_p$ is in $M_k^\#(\Gamma_0(Np))$. 
\end{Lemma} 

\begin{proof} Suppose $f|V_p=f(pz)$ has a pole at a rational cusp $x\over p$. Then $f$ has a pole at $x$, so $x$ is equivalent to the cusp at infinity under $\Gamma_0(N)$. In particular, we can write $x = \frac{a}{c}$ for $(a, c) = 1$ with $N \mid c$. Since $p \mid N$, we have $(a, pc) = 1$ so there exist integers $e$ and $f$ such that $\left(\begin{array}{cc} a & e \\ pc & f \end{array}\right) \in \Gamma_0(pN)$. This shows that $\frac{a}{pc} = \frac{x}{p}$ is equivalent to infinity in $\Gamma_0(Np)$, so $f \vert V_p$ has its only poles at infinity.
\end{proof}

\subsection{Proof of Theorem \ref{poly}} \label{proofpoly}

We begin with the following claim, after which the theorem follows directly using the same argument as in Theorem 3 of ~\cite{kaneko}.

\begin{Lemma} \label{woah}
Let $N$ be in the set $\{1, 2, 3, 4, 5, 6, 7, 8, 9, 10, 12, 13, 15, 18, 25\}$. For each positive integer $\nu$, the sum
\[\sum_{d|(\nu,N)} \frac{\nu}{d} J^{(N)}(z) \mid T_{\frac{\nu}{d}}V_d\]
is the unique meromorphic modular form in $M_0^{\#}(\Gamma_0(N))$ with Fourier expansion beginning $q^{-\nu} + O(q)$.
\end{Lemma}

\begin{proof}
The uniqueness of these forms is clear because the difference of any such modular forms is holomorphic of weight $0$ with zero constant term, and hence zero.

We first give a proof when the level $N = p$ is prime.
When $(\nu,p)=1$, Lemma \ref{Tqexp} shows that $\nu J^{(p)}(z)|T_{\nu} =q^{-\nu}+O(q)$ has the desired expansion at infinity. Then, Lemma \ref{winnie} shows that $\nu J^{(p)}(z)|T_{\nu}W_p=\nu J^{(p)}(z)|W_pT_{\nu}$. Since $W_p$ interchanges the poles at $0$ and $i \infty$ and $J^{(p)}(z)$ has no pole at zero, this shows that $\nu J^{(p)}(z)\vert T_{\nu}$ has no pole at zero. When $\nu$ is a multiple of $p$, two terms contribute to the sum, and we expand each at infinity and at zero. Write $\nu=p^as$ where $(s,N)=1$. At infinity, we have
\[\nu J^{(p)}(z)|T_{\nu} = \nu J^{(p)}(z)|U_p^aT_s=O(q)\]
because $J^{(p)}(z)$ has a simple pole which is killed by the $U$-operator. By Lemma \ref{Tqexp},
\[\frac{\nu}{p} J^{(1)}(z)|T_{\frac{\nu}{p}}V_p =q^{-\nu}+O(q).\]

To expand at zero, we apply the Atkin-Lehner involution at $p$. This gives
\begin{align*}
\nu J^{(p)}(z)|T_{\nu}W_p &=\nu J^{(p)}|U_p^aW_pT_s=s(q^{-p^{(a-1)}}+O(q))|T_s \\
&=q^{-p^{(a-1)s}}+O(q)=-q^{-\nu/p}+O(q)
\end{align*}
by applying first Lemma \ref{winnie} \eqref{WT=TW}, and then Corollary \ref{poles} followed by Lemma \ref{Tqexp}. Finally, by Lemmas \ref{VW} and  \ref{Tqexp}, we have
\[\frac{\nu}{p} J^{(1)}(z)|T_{\frac{\nu}{p}}V_pW_p =\frac{\nu}{p} J^{(1)}(z)|T_{\frac{\nu}{p}}=q^{-\nu/p}+O(q).\]
Adding together these terms shows we have the desired expansion at infinity and that the poles at zero cancel.

When $N$ is a perfect power of a prime, the lemma follows from the prime case after applying Lemma \ref{zero} and Lemma \ref{V}. 

We now treat the case when $N=p_1p_2$ is the product of two distinct primes. There are now four cusps (at $0, \frac{1}{p_1}, \frac{1}{p_2}$, and infinity) and we consider the expansion of each term in the sum at each cusp. Write $\nu=p_1^ap_2^bs$ where $(s,N)=1$. 
There will be up to four terms in the sum depending on $a$ and $b$: $d=1, d=p_1, d=p_2,$ and $d=N$.

At infinity, it is easy to see that when $d=(\nu, N)$, 
applying $T_{\frac{\nu}{d}}$ and then $V_d$ gives a pole of order $-\nu$ as desired, while for $d<(\nu,N)$, the Hecke operator $T_{\frac{\nu}{d}}=U_{(\frac{N}{d},\frac{\nu}{d})}T_s$ and the $U$ operator kills the pole at infinity.

At the cusp $\frac{1}{p_1}$, the four possible terms in the sum contribute in the following way when present:
\begin{align*}
 d = 1: \qquad \quad  \ \ \nu J^{(N)}(z)|T_{\nu}W_{p_2} &= \nu J^{(N)}(z)|U_{p_1}^aU_{p_2}^bT_sW_{p_2} =\nu J^{(N)}(z)|U_{p_2}^bW_{p_2}U_{p_1}^aT_s \\
&=\begin{cases} -q^{-\nu/p_2}+O(q) & \text{if  $b> 0$ and $a=0$} \\ O(q) & \text{otherwise} \end{cases} \\
d=p_1: \quad  \frac{\nu}{p_1} J^{(p_2)}(z)\vert T_{\frac{\nu}{p_1}}V_{p_1}W_{p_2} &= \frac{\nu}{p_1} J^{(p_2)}(z)|U_{p_2}^bT_{{p_1}^{a-1}s}V_{p_1}W_{p_2} \\
&=\frac{\nu}{p_1} J^{(p_2)}(z)|U_{p_2}^bW_{p_2}T_{{p_1}^{a-1}s}V_{p_1}\\
&=\begin{cases} -q^{-\nu/p_2} + O(q) & \text{if } b>0\\ O(q) & \text{if } b=0\end{cases}\\
d=p_2:\quad  \frac{\nu}{p_2}J^{(p_1)}(z)| T_{\frac{\nu}{p_2}}V_{p_2}W_{p_2} &=\frac{\nu}{p_2}J^{(p_1)}(z)|T_{\frac{\nu}{p_2}} = \frac{\nu}{p_2}J^{(p_1)}(z) \vert U_{p_1}^aT_{p_2^{b-1}s} \\
  &=\begin{cases} q^{-\nu/p_2} + O(q) & \text{if $a = 0$}\\ O(q) & \text{if $a > 0$}\end{cases}\\
d = N: \quad \ \  \frac{\nu}{N} J^{(1)}(z)|T_{\frac{\nu}{N}}V_{N}W_{p_2} 
&=\frac{\nu}{N}J^{(1)}(z)|T_{\frac{\nu}{N}}V_{p_2}W_{p_2}V_{p_1} =\frac{\nu}{N}J^{(1)}(z)|T_{\frac{\nu}{N}}V_{p_1} \\&=q^{-\nu/p_2} + O(q)
\end{align*}
For any choice of $a$ and $b$, the negative terms in the $q$-expansions of the different terms present cancel each other, showing that the sum has no pole at $\frac{1}{p_1}$. By symmetry, the same is true at the cusp $\frac{1}{p_2}$. It remains to check the expansion at zero.

Applying operators as before, we find
\begin{align*}
d=1: \qquad \quad \ \ \nu J^{(N)}(z)|T_{\nu}W_{N} &= \nu J^{(N)}(z)|U_{p_1}^aU_{p_2}^bT_sW_{p_1}W_{p_2}  \\
&= \nu J^{(N)}(z)|U_{p_1}^aW_{p_1}U_{p_2}^bW_{p_2}T_s \\
&=\begin{cases} O(q) & \text{if $a=0$ or $b=0$} \\ q^{-\nu/N} + O(q) & \text{if $a>0$ and $b>0$} \end{cases} \\
d=p_1: \quad  \frac{\nu}{p_1}J^{(p_2)}(z)|T_{\frac{\nu}{p_1}}V_{p_1}W_{N} &=\frac{\nu}{p_1}J^{(p_2)}(z)|U_{p_2}^bT_{p_1^{a-1}s}V_{p_1}W_{p_1}W_{p_2} \\
&=\frac{\nu}{p_1}J^{(p_2)}(z)|U_{p_2}^bW_{p_2}T_{p_1^{a-1}s}\\
&=\begin{cases} O(q) & \text{if $b=0$} \\ -q^{-\nu/N} + O(q) & \text{if $b>0$} \end{cases}\\
d=p_2: \quad \frac{\nu}{p_2}J^{(p_1)}(z)|T_{\frac{\nu}{p_2}}V_{p_2}W_{N} &= \frac{\nu}{p_2}J^{(p_1)}(z)|U_{p_1}^aT_{p_2^{b-1}s}V_{p_2}W_{p_2}W_{p_1} \\
&=\frac{\nu}{p_2}J^{(p_1)}(z)|U_{p_1}^aW_{p_1}T_{p_2^{b-1}s}\\
&=\begin{cases} O(q) & \text{if $a=0$} \\ -q^{-\nu/N} +O(q)& \text{if $a>0$} \end{cases} \\
d=N: \quad \ \  \frac{\nu}{N} J^{(1)}(z)|T_{\frac{\nu}{N}}V_NW_N &= q^{-\nu/N} + O(q).
\end{align*}
Again, for any choice of $a$ and $b$, the negative terms in the $q$-expansions of the different terms present cancel each other, showing that the sum has no pole at infinity.

This proves the lemma for all $N$ except for $N=12$ and $N=18$. Write $N = p_1^2p_2$. When $p_1|\nu$, applying Lemma \ref{zero} we can write 
\[\sum_{d|(\nu,N)} \frac{\nu}{d}J^{({N}/{d})}(z)|T_{\frac{\nu}{d}}V_d=p_1\left(\sum_{d|(\nu',N')}\frac{\nu'}{d}J^{(N'/d)}(z)|T_{\frac{\nu'}{d}}V_d\right)\vert V_{p_1}\]
 where $\nu'=\frac{\nu}{p_1}$ and $N'=p_1p_2$ and so the lemma follows from the case when $N$ is the product of two distinct primes together with Lemma \ref{V}.

If $p_1\nmid \nu$, the only terms contributing to the sum over divisors of $(\nu,N)$ are $d=1$ and $d=p_2$. The expansion at infinity is as claimed by the same argument as before. Expanding at each of the other cusps as before, we find there are no poles in each term except at the cusp at $\frac{1}{p_1^2}$. Here we find
\begin{align*}
d=1: \quad  \nu J^{(N)}(z)|T_{\nu}W_{p_2}=\nu J^{(N)}(z)|U_{p_2}^bW_{p_2}T_s&=\begin{cases} O(q) & \text{if $b = 0$} \\ -q^{-\nu/p_2}+O(q) & \text{if $b > 0$}\end{cases} \\
d=p_2: \quad \ \ \frac{\nu}{p_2}J^{(p_1^2)}(z)|T_{\frac{\nu}{p_2}}V_{p_2}W_{p_2}=\frac{\nu}{p_2}J^{(p_1^2)}|T_{\frac{\nu}{p_2}}&=q^{-\nu/p_2}+O(q),
\end{align*}
showing that the sum has no poles except at infinity.
\end{proof}

The following lemma is stated in ~\cite{kaneko} for level $1$ and weights $k$ such that $M_k(\Gamma_0(1))$ has dimension $1$. Given Lemma \ref{woah}, the proof for weight $0$ and levels $N$ where a Hauptmodul exists follows identically.

\begin{Lemma} \label{kan}
Let $N$ be in the set $\{1, 2, 3, 4, 5, 6, 7, 8, 9, 10, 12, 13, 15, 18, 25\}$. For each positive integer $\nu$, let $f_{\nu}^{(N)}(z)$ be the unique meromorphic modular form in $M_{0}^{\#}(\Gamma_0(N))$ with Fourier expansion beginning $q^{-\nu} + O(q)$. Then
\[\frac{j^{(N)'}(q)}{j^{(N)}(p) - j^{(N)}(q)} = \sum_{\nu=0}^{\infty}f_{\nu}^{(N)}(p)q^{\nu},\]
where $p$ and $q$ are independent formal variables and $f_0^{(N)}(z) = 1$.
\end{Lemma}

\begin{proof}
First we set
\begin{equation} \label{js}
j^{(N)}(q) = q^{-1} + \sum_{n=0}^{\infty}a_n^{(N)}q^n \qquad \text{and} \qquad J^{(N)}(q) = q^{-1} +  \sum_{n=1}^{\infty}a_n^{(N)}q^n.
\end{equation}
It is clear from the description of $f_{\nu}^{(N)}(z)$ in Lemma \ref{woah}, as
\[f_{\nu}^{(N)}(z) = \sum_{d|(\nu,N)} \frac{\nu}{d}J^{(N)}(z) \mid T_{\frac{\nu}{d}}V_d,\]
 that the coefficient of $q$ comes from the term where $d = 1$, and using \eqref{heckeformula} we see
 \begin{equation} \label{fnu}
 f_{\nu}^{(N)}(q) = q^{-\nu} + \nu a_{\nu}^{(N)}q + \ldots.
 \end{equation}
Since the form $j^{(N)}(p)f_{\nu}^{(N)}(p)$ is uniquely determined by the non-positive terms in its Fourier expansion, using \eqref{js} and \eqref{fnu} we obtain the recurrence relation
\begin{equation*}
j^{(N)}(p)f_{\nu}^{(N)}(p) = f_{\nu+1}^{(N)}(p) + \sum_{\ell = 0}^{\nu}a_{\nu - \ell}^{(N)}f_{\ell}^{(N)}(p) + \nu a_{\nu}^{(N)}.
\end{equation*}
Multiplying both sides of this equation by $q^{\nu}$ and summing over all $\nu$ gives us
\begin{align*}
j^{(N)}(p)F(p, q) &= \frac{1}{q}(F(p, q) - 1) + \left(j^{(N)}(q) - \frac{1}{q}\right)F(p, q) + \sum_{\nu = 0}^{\infty}\nu a_{\nu}^{(N)}q^{\nu} \\
&= j^{(N)}(q)F(p, q) + j^{(N)'}(q)
\end{align*}
where $F(p, q) = \sum_{\nu=0}^{\infty}f_{\nu}^{(N)}(p)q^{\nu}$. Rearranging gives the desired expression.
\end{proof}

Theorem \ref{poly} now follows immediately from these two results.
\begin{proof}[Proof of Theorem 1.1]
It follows from Lemma \ref{kan} that $P_{\nu}^{(N)}(j^{(N)}(z))$ is the unique modular form in $M_{0}^{\#}(\Gamma_0(N))$ with Fourier expansion beginning $q^{-\nu} + O(q)$. Then, applying Lemma \ref{woah} proves the identity.
\end{proof}

\section{Half integral weight modular forms} \label{forms}
In this section, we consider only levels $4N$ where $N$ is odd and square-free. We use the results of Miller and Pixton ~\cite{MP} to show that the traces $\text{Tr}_{\nu}^N(D)$ are expressible in terms of coefficients of certain weakly holomorphic modular forms of weight $\frac{3}{2}$.
For positive integers $\lambda, \nu, N$ the integral weight Niebur Poincar\'{e} series are defined by \cite{N}
$$\mathfrak{F}_{\lambda, N,\nu}(z):= \pi  \nu^{\lambda - 1}\sum_{A\in \Gamma_{\infty}\backslash \Gamma_0(N)} \text{Im}(\nu Az)^{\frac{1}{2}}I_{\lambda - \frac{1}{2}}(2\pi \text{Im}(\nu Az))e(-\text{Re}(\nu Az)),$$ where $I_s(x)$ is the usual modified Bessel function of the first kind, and $\Gamma_{\infty} : = \left\{\left(\begin{array}{cc} 1 & n \\ 0 & 1\end{array}\right) : n \in \mathbb{Z} \right\}$ denotes the stabilizer of infinity. They then construct half-integal weight Poincar\'{e} series whose coefficients describe the traces of the $\mathfrak{F}_{\lambda, N, \nu}(z)$. 

We first relate $\mathfrak{F}_{1, N, \nu}(z)$ to degree $\nu$ polynomials in the Hauptmodul for level $N$. Let $\mu(n)$ be the  M\"{o}bius function, defined by
\[\mu(n) = \begin{cases} (-1)^{t} & \text{if $n$ is square-free with $t$ prime factors} \\ 0 & \text{if $n$ is divisible by a square}.\end{cases}\]
It is not difficult to show that $\mu(n)$ is equal to the sum of the primitive $n$th roots of unity:
\[\mu(n) = \sum_{v \! \! \! \! \pmod{n}^*}e^{2\pi i v/n},\]
where the * indicates that the sum is taken over primitive residue classes mod $n$.
Let $\varphi(n)$ be Euler's totient function, and let $\zeta(s) = \sum_{c > 0} \frac{1}{c^s}$ be the Riemann-zeta function. The following identities will be useful for explicitly computing the constant term $\mathfrak{c}_{N, \nu}$ that appears in the subsequent Lemma \ref{relate}.

\begin{Lemma} \label{cool}
For any positive integer $n > 1$ and any real number $s > 1$, the following are true:
\begin{align*}
\sum_{c > 0} \frac{\mu(c)}{c^s} &= \frac{1}{\zeta(s)}
\intertext{and}
\sum_{c > 0} \frac{\mu(nc)}{(nc)^s} &= \frac{\mu(n)}{n^s - 1}\sum_{\substack{d \mid n \\ d \neq n}} \mu(d) \left(\sum_{c>0} \frac{\mu(cd)}{(cd)^s}\right).
\end{align*}
\end{Lemma}

\begin{proof}
The first identity is well-known and is proved by observing that the product of the two Dirichlet series $\zeta(s)\sum_{c > 0} \frac{\mu(c)}{c^s} = 1$ (because the convolution of $1$ and $\mu(c)$ is $1$ for $c=1$ and $0$ for all other $c$).

To prove the second identity, we use the inclusion/exclusion principle to write
\begin{align*}
\sum_{c > 0} \frac{\mu(nc)}{(nc)^s} &= \sum_{(c, n) = 1}\frac{\mu(nc)}{(nc)^s} = \frac{\mu(n)}{n^s}\sum_{(c, n)=1}\frac{\mu(c)}{c^s} \\
&= \frac{\mu(n)}{n^s}\sum_{\substack{d \mid n \\ d \neq n}} \mu(d)\left(\sum_{c>0}\frac{\mu(dc)}{(dc)^s}\right) + \frac{1}{n^s}\sum_{c>0}\frac{\mu(nc)}{(nc)^s.}
\end{align*}
Rearranging gives the recursive identity.
\end{proof}

We now explicitly describe the relationship between Miller and Pixton's integral weight Poincar\'e series and special polynomials in Haupmoduln.
\begin{Lemma} \label{relate}
For $N \in \{1, 3, 5, 7, 13\}$, and any positive integer $\nu$, we have
\[2\mathfrak{F}_{1, N, \nu}(z) = P_{\nu}^{(N)}(j^{(N)}(z)) + \mathfrak{c}_{N, \nu}, \]
where $P_{\nu}^{(N)}(x)$ is the polynomial defined in equation \eqref{defP} and 
\[\mathfrak{c}_{N, \nu} = 
\begin{dcases} 4\nu \pi^2 \sum_{d | \nu}\left(\sum_{\ell \mid d}
 \frac{\varphi(d)\ell}{\varphi(\ell)d^2}
\left( \sum_{x \mid \frac{\nu}{d}, \ y \mid \frac{d}{\ell}} \mu(x)\mu(y)
\left(\sum_{c > 0} \frac{\mu(\frac{Nxy\ell}{(x, y\ell)}c)}{(\frac{Nxy\ell}{(x, y\ell)}c)^2} \right)\right)\right)
&\text{if $N \nmid \nu$} \\
4\nu \pi^2 \sum_{\substack{d | \nu \\ N \mid d}} 
\left(\sum_{\substack{\ell \mid d \\ N \mid \ell}} \frac{\varphi(d)\ell}{\varphi(\ell)d^2}
\left(\sum_{x \mid \frac{\nu N}{d}, \ y \mid \frac{Nd}{\ell}} \mu(x)\mu(y)
\left(\sum_{c > 0} \frac{\mu(\frac{xy\ell}{(x, y\ell)}c)}{(\frac{xy\ell}{(x, y\ell)}c)^2}\right)\right)\right) & \text{if $N \mid \nu$.}
\end{dcases}\]
In particular, $\mathfrak{c}_{N, \nu}$ is a rational number.
\end{Lemma}

\begin{proof}
By Theorem 1.2 of ~\cite{MP}, $2\mathfrak{F}_{1, N, \nu}(z)$ differs from $P_{\nu}^{(N)}(j^{(N)}(z))$ by a constant. Since the constant term in the Fourier expansion of $P_{\nu}^{(N)}(j^{(N)}(z))$ is zero, the constant $\mathfrak{c}_{N, \nu}$ is equal to twice the constant term of $\mathfrak{F}_{1, N, \nu}(z)$. Using equation (17) of ~\cite{MP}, we determine that
\[\mathfrak{c}_{N, \nu} = 4\nu\pi^2 \sum_{c > 0} \frac{\mathcal{S}(0, -\nu; Nc)}{(Nc)^2},\]
where $\mathcal{S}(0, -\nu;c)$ is the Kloosterman sum
\[\mathcal{S}(0, -\nu; c) := \sum_{v \! \! \! \pmod{c}^*} e^{2\pi i \nu v/c}.\]
We re-write the Kloosterman sum in terms of the M\"{o}bius function as
\[ \mathcal{S}(0, -\nu; c) = \frac{\varphi(c)}{\varphi(\frac{c}{(c, \nu)})}\mu\left(\frac{c}{(c, \nu)}\right),\]
so that the constant term becomes
\begin{align*}
\mathfrak{c}_{N, \nu} &= 4\nu\pi^2\sum_{d \mid \nu} \left(\sum_{\substack{c > 0 \\ (Nc, \nu) = d}}\frac{\varphi(Nc)}{\varphi(\frac{Nc}{d})d^2} \cdot \frac{\mu(\frac{Nc}{d})}{(\frac{Nc}{d})^2}\right) \\
&= 4\nu\pi^2\sum_{d \mid \nu}
\left(\sum_{\ell \mid d} \frac{\varphi(d)\ell}{\varphi(\ell)d^2}
\left(\sum_{\substack{c > 0\\(Nc, \nu) = d \\ (d, \frac{Nc}{d}) = \ell}}
\frac{\mu(\frac{Nc}{d})}{(\frac{Nc}{d})^2} \right)\right).
\end{align*}
Recall $N = 1$ or a prime. The set $\{c > 0 : (Nc, \nu) = d\}$ can be written as
\begin{align}
\{c &> 0: (Nc, \nu) = d\} \notag \\
&=
\begin{dcases} \{c > 0 : (c, \nu) = d\} = \sum_{x \mid \frac{\nu}{d}} \mu(x)\{xdc : c > 0\} &\text{if $N \nmid \nu$} \\
\{c > 0: (c, \nu) = \tfrac{d}{N}\} = \sum_{x \mid \frac{\nu N}{d}}\mu(x)\{\tfrac{xdc}{N} : c > 0\} & \text{if $N \mid \nu$ and $N \mid d$} \\
\varnothing & \text{otherwise.} \label{decomp}
\end{dcases}
\end{align}
The set $\{c > 0: (\tfrac{Nc}{d}, d) = \ell\}$ can be described in a similar way. Taking the intersection of these sets yields
\begin{align*}
\{c >0 : (Nc, \nu) = d,  (\tfrac{Nc}{d}, d) = \ell\} = \begin{dcases}
\sum_{x \mid \frac{\nu}{d}, \ y \mid \frac{d}{\ell}} \mu(x)\mu(y)\{\tfrac{dxy\ell}{(x, y\ell)}c : c > 0\} & \text{if $N \nmid \nu$} \\
\sum_{x \mid \frac{\nu N}{d}, \ y \mid \frac{Nd}{\ell}} \mu(x)\mu(y)\{\tfrac{dxy\ell}{N(x, y\ell)}: c > 0 \} & \text{if $N \mid \nu, d, \ell$} \\
\varnothing & \text{otherwise.}
\end{dcases}
\end{align*}
Using this decomposition to re-write the inner sum gives the desired expression.

Since $\frac{1}{\zeta(2)} = \frac{6}{\pi^2}$, applying Lemma \ref{cool} shows that the constant term $\mathfrak{c}_{N, \nu}$ is always rational.
\end{proof}

We now introduce Miller and Pixton's sequence of half-integal weight Poincar\'{e} series, whose coefficients will describe traces of polynomials in Hauptmoduln. We restrict our attention to the Poincar\'{e} series of weight $\frac{3}{2}$, as the coefficients of these forms are sufficient to determine the traces of polynomials in Hauptmoduln. However, we note that, due to the duality properties relating the forms of weight $\frac{3}{2}$ and forms of weight $\frac{1}{2}$ (see Corollary 1.4 of ~\cite{MP}), this could also be accomplished by working with forms of weight $\frac{1}{2}$.

Following Miller and Pixton, for $s \in \mathbb{C}$ and $y \in \mathbb{R} - \{0\}$, we define
\[\mathcal{M}_{s}(y):= |y|^{-\frac{3}{4}}M_{\frac{3}{4}\text{sgn}(y), s - \frac{1}{2}}(|y|),\]
where $M_{\nu, \mu}(z)$ denotes the usual $M$-Whittaker function. For $m \geq 1$ with $m \equiv 0, 1 \pmod{4}$, let
\[\varphi_{-m, s}(z) := \mathcal{M}_s(-4\pi m \mathrm{Im}(z))e(-m\mathrm{Re}(z)).\]
Now, define the Poincar\'e series of weight $\frac{3}{2}$
$$\mathcal{F}_N(-m, s;z):=\sum\limits_{A\in \Gamma_{\infty}\backslash\Gamma_0(4N)} (\varphi_{-m, s}|_{\frac{3}{2}}A)(z),$$ 
which converges for $\mathrm{Re}(s) > 1$.

\begin{remark}
The slash operator for half-integral weight modular forms requires a slight modification. For $A = \left(\begin{array}{cc} \alpha & \beta \\ \gamma & \delta \end{array}\right) \in \Gamma_0(4)$, we define
\[(f \vert_{k} A) := \left(\left( \frac{\gamma}{\delta}\right)\epsilon_{\delta}^{-1}(\gamma z + \delta)^{\frac{1}{2}} \right)^{-2k}f(Az),\]
where
\[\epsilon_{\delta} := \begin{cases}1 & \text{if $ \delta \equiv 1 \pmod{4}$} \\ i & \text{if $\delta \equiv 3 \pmod 4.$} \end{cases}\]
\end{remark}

For the special value $s = \frac{3}{4}$, we define 
$$F_{N}(-m;z):=\frac{3}{2}\mathcal{F}_{N}(-m, \tfrac{3}{4}; z)|\text{pr}_1$$
to be the projection of $\mathcal{F}_{N}(-m, \frac{3}{4};z)$ into the plus space using Kohnen's projection operator $\text{pr}_{\lambda}$ (defined in \cite{kohnen}). Note that this definition requires analytic continuation.
The $F_{N}(-m;z)$ are weak Maass forms. Let $\varpi(N)$ be the index of $\Gamma_0(N)$ in the full modular group, and let $\beta(s):=\int\limits_1^{\infty} t^{-3/2}e^{-st}dt$. Set
\[\delta_{\square}(m) := \begin{cases} 1 & \text{if $m$ is square,} \\ 0 & \text{otherwise.}\end{cases}\]
 Theorem 2.1 of ~\cite{MP} shows that $F_N(-m;z)$ has a Fourier expansion of the form
\begin{equation} \label{Fexpansion}
F_{N}(-m;z)=q^{-m} + \! \! \! \! \sum\limits_{\substack{ n\geq 0 \\ n \equiv 0, 3 \! \! \! \! \pmod4}} \! \! \! \! b_{N}(-m;n)q^n
-\frac{3\delta_{\square}(m)}{2\pi\varpi(N)\sqrt{y}} \sum_{n = -\infty}^{\infty} \beta(4\pi n^2 y)q^{-n^2}
\end{equation}
and gives formulas for the coefficients $b_N(-m;n)$ as infinite sums of Kloosterman sums weighted by Bessel functions. The next lemma relates the forms $F_{N}(-m;z)$ to a family of weakly holomorphic modular forms.

Let $H_N(D)$ be the \textit{generalized class numbers} defined by
\begin{equation}
H_N(D) := \sum_{Q \in \mathcal{Q}_D^N/\Gamma_0(N)} \frac{1}{w_{Q, N}} \qquad \text{and} \qquad H_N(0) := \frac{-\varpi(N)}{12}.
\end{equation}
The numbers $H_{1}(D)$ are well-known to be coefficients of Zagier's non-holomorphic Eisenstein series of weight $\frac{3}{2}$, defined by
\begin{equation} \label{defG}
G(z):=\sum\limits_{n=0}^{\infty} H_1(n)q^n+ \frac{1}{16\pi\sqrt{y}}\sum\limits_{n=-\infty}^{\infty}\beta(4\pi n^2y)q^{-n^2}.\end{equation}

\begin{Lemma} \label{defFtilde}
Let $N \in \{1, 3, 5, 7, 13\}$ and let $m$ be a positive integer satisfying $m \equiv 0, 1 \pmod{4}$. Then
\begin{equation*}
\widetilde{F}_N(-m;z) := F_N(-m;z) + \frac{24\delta_{\square}(m)}{\varpi(N)}\sum_{n=0}^{\infty}H_1(n)q^n
\end{equation*}
is a weakly holomorphic modular form of weight $\frac{3}{2}$ on $\Gamma_0(4N)$. If $N > 1$, then
\[\widetilde{F}_N(0;z) := \sum_{n=0}^{\infty}(2H_1(n) - H_N(n))q^n\]
is a holomorphic modular form of weight $\frac{3}{2}$ on $\Gamma_0(4N)$. 
\end{Lemma}

\begin{proof}
Looking at equations \eqref{Fexpansion} and \eqref{defG}, we see that the non-holomorphic parts of $F_N(-m;z)$  and $\frac{24\delta_{\square}(m)}{\varpi(N)}G(z)$ cancel each other, so $F_N(-m;z) + \frac{24\delta_{\square}(m)}{\varpi(N)}G(z)$ is a weakly holomorphic modular form equal to $\widetilde{F}_N(-m;z)$.

Arguing as in Chapter 2 of ~\cite{HZ}, one can confirm modularity properties of the $H_N(D)$ for general $N$ by finding further non-holomorphic Eisenstein series with coefficients equal to $H_N(D)$ whose non-holomorphic parts are also period integrals of the $\theta$-function. Subtracting these forms from $2G(z)$ cancels the non-holomorphic parts, leaving a holomorphic form equal to $\widetilde{F}_N(0;z)$.
\end{proof}

Let $\widetilde{b}_N(-m;n)$ be the coefficient of $q^n$ in $\widetilde{F}_N(-m;z)$ so that
\[\widetilde{F}_N(-m;z) = q^{-m} + \sum_{n \equiv 0, 3 \! \! \! \! \pmod{4}}\widetilde{b}(-m;n)q^n.\]

The following theorem describes traces of polynomials in Hauptmoduln in terms of the coefficients $\widetilde{b}_N(-m;n)$. It is a more-explicit version of Theorem 1.1 of \cite{MP}, and generalizes Theorem 1.2 of \cite{BO}.
This makes it possible to calculate these traces from the weakly holomorphic modular forms $\widetilde{F}_N(-m; z)$ given in the Appendix.

\begin{Theorem}
Let $N \in \{1, 3, 5, 7, 13\}$. For each positive integer $D$ with $D \equiv 0, 3 \pmod 4$, we have
\[\mathrm{Tr}_{\nu}^{(N)}(D) =-\nu \sum\limits_{d|\nu} \frac{1}{d}\left(\widetilde{b}_{\frac{N}{(N,d)}}(\tfrac{-\nu^2}{d^2};D) - \frac{24}{\varpi(N)}H_{1}(D)\right) - H_N(D)\mathfrak{c}_{N, \nu}.\]
Furthermore, $\widetilde{F}_N(0;z)$ and $\widetilde{F}_N(-1;z)$ satisfy the identities listed in the Appendix.
\end{Theorem}

\begin{proof}
Using Theorem 1.1 of ~\cite{MP} together with Lemma \ref{relate}, we find
\begin{align*}
\mathrm{Tr}_{\nu}^{(N)}(D) &=
\sum_{Q \in \mathcal{Q}_D^N/\Gamma_0(N)} \! \! \frac{P_{\nu}^{(N)}(j^{(N)}(\alpha_Q)) }{w_{Q, N}} \\
&= 2 \! \! \sum_{Q \in \mathcal{Q}_D^N/\Gamma_0(N)} \frac{\mathfrak{F}_{1, N, \nu}(\alpha_Q)}{w_{Q, N}} -  \! \! \sum_{Q \in \mathcal{Q}_D^N/\Gamma_0(N)} \! \! \frac{\mathfrak{c}_{N, \nu}}{w_{Q, N}} \\
&= -\nu \sum\limits_{d|\nu} \frac{1}{d}b_{\frac{N}{(N,d)}}(\tfrac{-\nu^2}{d^2};D) - H_N(D)\mathfrak{c}_{N, \nu},
\end{align*}
Applying the definition of $\widetilde{F}_N(-m;z)$ in Lemma \ref{defFtilde} gives the above formula. 

Hauptmoduln take on algebraic-integer values at Heegner points, so when $\nu = 1$ the left-hand side of the formula for the trace is $\frac{1}{2}$ or $\frac{1}{3}$ times an integer. Since the constants appearing on the right-hand side are rational with bounded denominator, so are the coefficients $\widetilde{b}_N(-1;n)$. Since the Kloosterman sums in Miller and Pixton's Theorem 2.1 converge, the coefficients of $\widetilde{F}_N(-1;z),$ can be determined from these formulas by taking sufficiently large partial sums. A standard argument then shows that the identities in the Appendix hold after comparing finitely many coefficients. Since $\widetilde{F}_N(-1;z)$ is a weakly holomorphic modular form, multiplying by an appropriate cusp form, for example $\eta(4z)^{6}$, lands the product in a space of holomorphic modular forms. These spaces are finite-dimensional, so one need only check that finitely many coefficients in their $q$-expansions agree.
The number of coefficients one has to check is determined by the dimension-formulas given in Theorem 1.34 of ~\cite{ken}.
\end{proof}

\begin{remark}
It should be pointed out that the theory of half-integral weight modular forms is particularly difficult for levels $4N$ when $N$ is even or not-square free. However, for $N=2$ we observe that the definitions in Lemma \ref{defFtilde} still give weakly holomorphic modular forms (equal to those listed in the Appendix), and the formula for the traces in the above theorem still holds.
\end{remark}

\section{Algorithm for Computing $\mathcal{H}_D^{(N)}(x)$} \label{algsec}
Recall that for $\alpha_Q$ a Heegner point of level $N$ and discriminant $-D$, the minimal polynomial of $j^{(N)}(\alpha_Q)$  is given by 
$$\mathcal{H}_D^{(N)}(x)=\prod_{Q\in\mathcal{Q}_D/\Gamma_0(N)} (x-j^{(N)}(\alpha_Q)).$$ 
It is clear from this description that the coefficients of $\mathcal{H}_D^{(N)}(x)$ are the elementary symmetric polynomials in the $j^{(N)}(\alpha_Q)$. The following algorithm explains how to compute these class polynomials for $N = 1, 2, 3, 5, 7, 13$ and $-D$ a fundamental discriminant.

\vspace{.1in}
\noindent
\textbf{Algorithm}
\begin{enumerate}
\item Recursively generate the weakly holomorphic modular forms $\widetilde{F}_N(-m; z)$ for $1 \leq m \leq |\mathcal{Q}_D^N/\Gamma_0(N)|^2$ using the forms $\widetilde{F}_{N}(0;z)$ and $\widetilde{F}_N(-1;z)$ given in the Appendix.\label{comp}
\item Use Theorem \ref{traces} to calculate the traces $\text{Tr}_{\nu}^{(N)}(D)$ for $1 \leq \nu \leq |\mathcal{Q}_D^N/\Gamma_0(N)|$ by plugging in the appropriate coefficients $\widetilde{b}_N(-m;D)$ and constants $H_N(D)$, $H_1(D),$ and $\mathfrak{c}_{N, \nu}$ (and multiplying by $2$ or $3$ if $D = 3$ or $4$ respectively). \label{plugin}
\item Use the generating function in equation \eqref{defP} to obtain the coefficients of $P_{\nu}^N(x)$ for each $1 \leq \nu \leq |\mathcal{Q}_D^N/\Gamma_0(N)|$, and diagonalize to determine the values of the power sums $\sum (j^{(N)}(\alpha_Q))^{\nu}$.
\item Apply the Newton-Girard formulae
to recursively determine the elementary symmetric polynomials in $j^{(N)}(\alpha_Q)$. 
\end{enumerate}

\begin{proof}[Proof of Algorithm]
Step \ref{comp} can be accomplished by multiplying $\widetilde{F}_N(-m+4,z)$ by $j^{(N)}(4z)$ to obtain a non-trivial linear combination of $\widetilde{F}_N(-m;z)$ and $\widetilde{F}_N(-n;z)$ for $n<m$. Subtracting off appropriate multiples of $\widetilde{F}_N(-n;z)$ for $0 \leq n < m$ leaves a form with the same non-positive Fourier coefficients as $\widetilde{F}_N(-m;z)$. The two forms must then be equal because their difference is in $M_{\frac{3}{2}}(\Gamma_0(4N))$ with a zero constant term and hence zero.

In step \ref{plugin}, the constant $\mathfrak{c}_{N, \nu}$ can be computed from the formula in Lemma \ref{relate}, using the identities in Lemma \ref{cool} to express the infinite sums in terms of reciprocals of the Riemann-zeta function.

Finally, given the $i$th power sums $p_i(x_1, \ldots, x_n) = x_1^i + \ldots + x_n^i$, the Newton-Girard formulae recursively determine the elementary symmetric polynomials $e_k(x_1, \ldots, x_n)$ by
\[ke_k(x_1, \ldots, x_n) = \sum_{i=1}^k(-1)^{i-1}e_{k-i}(x_1, \ldots, x_n)p_i(x_1, \ldots, x_n). \qedhere \]
\end{proof}

\section{Explicit Numerical Examples}
We show how to compute the class polynomial for level $7$ and discriminant $-20$.
The following table lists representatives of binary quadratic forms for this level and discriminant along with an approximation (which required 1500 coefficients) of the Hauptmodul $j^{(7)}(z)$ evaluated at their roots:
\begin{center}
\begin{tabular}{|c|c|}
\hline
$Q\in\mathcal{Q}^7_{20}/\Gamma_0(7)$ & approximation of $j^{(7)}(\alpha_Q)$\\
\hline
$14x^2 + 6xy + y^2$ & $-4.1458\ldots + i1.2360\ldots $\\
\hline
$21x^2 + 8xy + y^2$ &  $-4.1458\ldots - i1.2360\ldots $\\
\hline
$7x^2 + 6xy + 2y^2$ & $-10.8541\ldots + i3.2360\ldots $\\
\hline
$63x^2 + 22xy + 2y^2$ & $-10.8541\ldots -i 3.2360\ldots $ \\
\hline
\end{tabular}
\end{center}

Using the seed forms given in the Appendix, we recursively compute the Fourier expansions of $\widetilde{F}_7(-m;z)$ for $m \leq 16$ by the method described in the proof of the algorithm.
In particular, we find the coefficients
\[\widetilde{b}_7(-1;20) = 22 \qquad \widetilde{b}_7(-4;20) = -26 \qquad \widetilde{b}_7(-9;20) = 78 \qquad \widetilde{b}_7(-16;20) = 338.\]
Next, we compute the constants $\mathfrak{c}_{7, \nu}$ for $1 \leq \nu \leq 4$ using the formula in Lemma \ref{relate} along with the identities in Lemma \ref{cool} to evaluate the infinite sums:
\[\mathfrak{c}_{7,1}=-\tfrac{1}{2} \qquad \mathfrak{c}_{7,2}=-\tfrac{3}{2} \qquad \mathfrak{c}_{7,3}=-2 \qquad \mathfrak{c}_{7,4}=-\tfrac{7}{2}.\]
We have $\frac{24}{\varpi(7)}H_{1}(20) = 6$ and $H_{7}(20) = 4$. Applying the formula in Theorem \ref{traces}, we obtain the traces
\begin{align*}
\text{Tr}_{1}^{(7)}(20) =-14 \qquad \text{Tr}_{2}^{(7)}(20) =54 \qquad \text{Tr}_{3}^{(7)}(20) =-224 \qquad \text{Tr}_{4}^{(7)}(20) =-1266.
\end{align*}
The polynomials $P_{\nu}^{(7)}(x)$ can be determined explicitly by expanding equation \eqref{defP} in a formal power series over the polynomial ring in $x$ and taking the coefficient of $q^{\nu}$:
\begin{align*}
P_{1}^{(7)}(x) = x+4 \qquad P_{2}^{(7)}(x)& = x^2+8x+12 \qquad P_{3}^{(7)}(x) = x^3+12x^2+42x+16 \\
 P_{4}^{(7)}(x) &= x^4+16x^3+88x^2+160x+28
\end{align*}
The power sums can now be determined by diagonalization, giving
\begin{align*}
\sum_{Q \in \mathcal{Q}_{20}^7/\Gamma_0(7)}j_{(7)}(\alpha_Q) &= -30 &\qquad 
\sum_{Q \in \mathcal{Q}_{20}^7/\Gamma_0(7)}(j_{(7)}(\alpha_Q))^2 &= 246 \\
\sum_{Q \in \mathcal{Q}_{20}^7/\Gamma_0(7)}(j_{(7)}(\alpha_Q))^3 &=  -1980 &\qquad
\sum_{Q \in \mathcal{Q}_{20}^7/\Gamma_0(7)}(j_{(7)}(\alpha_Q))^4 &= 13454.
\end{align*}
Finally, the elementary symmetric polynomials can be recovered recursively using the Newton-Girard formulae, thus determining  the class polynomial:
$$\mathcal{H}_{20}^{(7)}(x)=x^4+30x^3 +327x^2-1470x+2401.$$
We factor this polynomial as
\begin{align*}
&\left(x + \tfrac{15}{2} - \tfrac{3\sqrt{5}}{2} - i\sqrt{2(3 - \sqrt{5})}\right)
\left(x + \tfrac{15}{2} - \tfrac{3\sqrt{5}}{2} + i\sqrt{2(3 - \sqrt{5})}\right) \\
&\qquad \qquad \qquad \times \left(x + \tfrac{15}{2} + \tfrac{3\sqrt{5}}{2} + i\sqrt{2(3+\sqrt{5})}\right)
\left(x + \tfrac{15}{2} + \tfrac{3\sqrt{5}}{2} - i\sqrt{2(3+\sqrt{5})}\right),
\end{align*}
confirming the numerical approximations at the beginning of the section.

\section{Appendix}

Here we give explicit closed formulas for the seed functions $\widetilde{F}_{N}(0;z)$ and $\widetilde{F}_{N}(-1;z)$ for $N = 1, 2, 3, 5, 7, 13$. We write $E_4(z)$ for the Eisenstein series of weight $4$ and $E_2(z)$ for the Eisenstein series of weight $2$. Also, let
\[\theta(z) = \sum_{n \in \mathbb{Z}}q^{n^2} \qquad \text{and} \qquad \theta_1(z) = \sum_{n \in \mathbb{Z}}(-1)^nq^{n^2}\]
be the standard theta functions.
When the space of cusp forms of weight $k$ and level $N$ is one-dimensional, we will write $S_{k}^{(N)}(z)$ for the unique normalized form in this space. We also require the following forms:
\begin{align*}
S_2^{(26, +)}(z)& &&\text{the sum of the two newforms of weight $2$ on $\Gamma_0(26)$} \\
S_2^{(26, -)}(z)& &&\text{the difference of the two newforms of weight $2$ on $\Gamma_0(26)$ which} \\
&&&\text{begins $-2q^2 + 4q^3 - 2q^5 + \ldots$} \\
S_4^{(13, 1)}(z) \ & &&\text{the newform of weight $4$ on $\Gamma_0(13)$ with rational coefficients} \\
S_4^{(13, +)}(z)& &&\text{the sum of the two newforms of weight $4$ on $\Gamma_0(13)$ with coefficients} \\
&&&\text{in the field with defining polynomial $x^2 - x - 4$.}
\end{align*}
We can now list the following forms:
\begin{align*}
\widetilde{F}_{2}(0; z) &=  \frac{24E_2(8z) - 14E_2(4z) + 3E_2(2z) - E_2(z)}{144 \theta(z)} \\ 
\widetilde{F}_{3}(0;z) &=  \frac{24E_2(12z) - 9E_2(6z) - 8E_2(4z) + 3E_2(3z) + 3E_2(2z) - E_2(z)}{72\theta(z)} \\
\widetilde{F}_{5}(0;z) &= 
\frac{40E_2(20z)- 15E_2(10z) + 5E_2(5z)- 8E_2(4z)  + 3E_2(2z)  -E_2(z)  + 24S_2^{(20)}(z)}{72\theta(z)} \\
\widetilde{F}_{7}(0;z) &= \frac{12E_2(28z) -11E_2(14z)+13E_2(7z)+4E_2(4z) -7E_2(2z) +E_2(z)}{24\theta(z)} \\
&\qquad  - \frac{\eta(2z)^7\eta(14z)^7}{2\theta(z)\eta(z)^3\eta(4z)^2\eta(7z)^3\eta(28z)^2} 
+ \frac{5\eta(z)\eta(4z)^2\eta(14z)^{17}}{2\theta(z)\eta(2z)^3\eta(7z)^7\eta(28z)^6} \\ \\
\widetilde{F}_{13}(0;z) &= \frac{104E2(52z)-39E_2(26z) + 13E_2(13z) - 8E_2(4z) + 3E_2(2z) - E_2(z)}{72\theta(z)} \\
&\qquad + \frac{2S_2^{(26, +)}(z) + 2S_2^{(26, -)}(2z) + S_2^{(52)}(z)}{3\theta(z)} 
\\
\widetilde{F}_{1}(-1;z) &= \theta_1(z) \frac{E_4(4z)}{\eta(4z)^6} \\
\widetilde{F}_{2}(-1;z) &= \theta_1(z) \frac{16E_4(8z) - E_4(4z)}{15\eta(4z)^6} + 16F_{2}(0; z) \\
\widetilde{F}_{3}(-1;z) &= \theta_1(z) \frac{81E_4(12z) - E_4(4z)}{80\eta(4z)^6} + 9F_{3}(0;z) 
\\
\widetilde{F}_{5}(-1;z) &= \theta_1(z) \frac{\eta(4z)^4}{\eta(20z)^2} + 5F_{5}(0;z) 
\end{align*}
\begin{align*}
\widetilde{F}_{7}(-1;z) &= \theta_1(z) \frac{2401E_4(28z) - E_4(4z) - 11760S_4^{(7)}(4z)}{2400\eta(4z)^6} + \frac{7}{2}F_{7}(0;z) \\
\widetilde{F}_{13}(-1;z) &= 
\frac{\theta_1(z)}{\eta(4z)^6}
\left(-\frac{137E_4(4z)}{14280} -\frac{ 2197E_4(52z)}{3570} + \frac{13(13E_2(52z) - E_2(4z))^2}{1152}\right) \\
&\qquad -\frac{\theta_1(z)}{\eta(4z)^6}\left( \frac{39S_4^{(13, 1)}(4z)}{14} + \frac{143S_4^{(13, +)}(4z)}{34}\right) + \frac{13}{7}F_{13}(0;z)
\end{align*}

\begin{remark}
There is no holomorphic modular form of weight $\frac{3}{2}$ when $N = 1$, so $\widetilde{F}_{1}(-4;z)$ is required to recursively generate the $\widetilde{F}_1(-m;z)$ for all $m$. Zagier explains how to obtain $\widetilde{F}_{1}(-4;z)$ from $\widetilde{F}_1(-1;z)$ in the discussion preceding Theorem 4 of ~\cite{traces}.
\end{remark}
\bibliographystyle{plain}
\bibliography{billy}{}

\begin{thebibliography}{10}

\bibitem{kaneko}
T.~Asai, M.~Kaneko, and H.~Ninomiya.
\newblock Zeros of certain modular functions and an application.
\newblock {\em Comment. Math. Univ. St. Paul.}, 46(1):93--101, 1997.

\bibitem{AK}
A.~O.~L. Atkin and J.~Lehner.
\newblock Hecke operators on {$\Gamma_{0}(m)$}.
\newblock {\em Math. Ann.}, 185:134--160, 1970.

\bibitem{BO}
K.~Bringmann and K.~Ono.
\newblock Arithmetic properties of coefficients of half-integral weight
  maass-poincar/'e series.
\newblock {\em Math. Ann.}, 337:591--612, 2007.

\bibitem{Suth}
A.~Enge and A.~Sutherland.
\newblock Class invariants by the crt method.
\newblock {\em Algorithmic Number Theory 9th International Symposium (ANTS
  IX)}, pages 142--156, 2010.

\bibitem{Gross}
B.~Gross.
\newblock The classes of singular moduli in the generalized jacobian.
\newblock {\em In: Geometry and arithmetic, Congress reports of the EMS}, pages
  137--142, 2012.

\bibitem{HZ}
F.~Hirzebruch and D.~Zagier.
\newblock Intersection numbers of curves on {H}ilbert modular surfaces and
  modular forms of {N}ebentypus.
\newblock {\em Invent. Math.}, 36:57--113, 1976.

\bibitem{kohnen}
W.~Kohnen.
\newblock Fourier coefficients of modular forms of half-integral weight.
\newblock {\em Math. Ann.}, 271(2):237--268, 1985.

\bibitem{Li}
W.~Li.
\newblock Newforms and functional equations.
\newblock {\em Math. Ann.}, 212:285--315, 1975.

\bibitem{MP}
A.~Miller and A.~Pixton.
\newblock Arithmetic traces of non-holomorphic modular invariants.
\newblock {\em Int. J. Number Theory}, 6(1):69--87, 2010.

\bibitem{N}
D.~Niebur.
\newblock A class of nonanalytic automorphic functions.
\newblock {\em Nagoya Math. J.}, 52:133--145, 1973.

\bibitem{ken}
K.~Ono.
\newblock {\em The web of modularity: arithmetic of the coefficients of modular
  forms and {$q$}-series}, volume 102 of {\em CBMS Regional Conference Series
  in Mathematics}.
\newblock Published for the Conference Board of the Mathematical Sciences,
  Washington, DC; by the American Mathematical Society, Providence, RI, 2004.

\bibitem{traces}
D.~Zagier.
\newblock Traces of singular moduli.
\newblock {\em Motives, polylogarithms and {H}odge theory, {P}art {I}},
  3:211--244, 2002.

\end{thebibliography}

\end{document}